\documentclass{article}

\usepackage{graphicx}
\usepackage[utf8]{inputenc}
\usepackage[english]{babel}
\usepackage{amsmath, amssymb, amsfonts, mathtools, xcolor, amsthm}
\usepackage{tikz}
\usepackage{scrextend}
\usepackage{url}
\usepackage{changepage}
\usepackage[colorlinks=false,linkcolor=red]{hyperref}
\usepackage{caption}
\usepackage{ifthen}
\usepackage{xstring}
\usepackage{xcolor}
\usepackage{hyperref}
\usepackage{makecell}
\usepackage[left=4cm,right=4cm]{geometry}
\usetikzlibrary{arrows}

\theoremstyle{plain}
\newtheorem{theorem}{Theorem}
\newtheorem{corollary}[theorem]{Corollary}
\newtheorem{proposition}[theorem]{Proposition}
\newtheorem{lemma}[theorem]{Lemma}

\theoremstyle{definition}

\newtheorem{remark}[theorem]{Remark}

\begin{document}

\title{An improved constant factor for the unit distance problem}

\author{P\'eter \'Agoston\footnote{MTA-ELTE Lend\"ulet Combinatorial Geometry (CoGe) Research Group, Budapest, Hungary and ELTE E\"otv\"os Lor\'and University, Budapest, Hungary, Faculty of Science. Supported by grants LP2017-19/2017 and EFOP-3.6.3-VEKOP-16-2017-00002.}\and D\"om\"ot\"or P\'alv\"olgyi\footnote{MTA-ELTE Lend\"ulet Combinatorial Geometry (CoGe) Research Group, E\"otv\"os Loránd University, Budapest, Hungary. Supported by grant LP2017-19/2017.}}

\maketitle

\begin{abstract}
We prove that the number of unit distances among $n$ planar points is at most $1.94\cdot n^{4/3}$, improving on the previous best bound of $8n^{4/3}$.
We also give better upper and lower bounds for several small values of $n$.
We also prove some variants of the crossing lemma and improve some constant factors.
\end{abstract}

\section{Introduction}

Call a simple graph a \emph{unit distance graph} (UDG) if its vertices can be represented by distinct points in the plane so that the pairs of vertices connected by an edge correspond to pairs of points at unit distance apart.
Denote the maximal number of edges in a unit distance graph with $n$ vertices by $u(n)$.
Erd\H os \cite{erdos} raised the problem to determine $u(n)$ and this question became known as the \emph{Erd\H os Unit Distance Problem}.
Erd\H os established the bounds $n^{1+c/\log\log n}\le u(n)\le O(n^{3/2})$.
The lower bound remained unchanged, but the upper bound has been improved several times, the current best has been $O(n^{4/3})$ for more than 35 years \cite{sszt}.
For a detailed survey, see \cite{sz3}.

It turned out during the Polymath16 project\footnote{See \url{https://dustingmixon.wordpress.com/2018/04/14/polymath16} and Section 4.3 in \url{https://arxiv.org/abs/2112.07665}.} that improved bounds even for small values of $n$ might give better bounds for questions related to the chromatic number of the plane.
Our goal is to give an explicit upper bound, thus a constant factor improvement of the $O(n^{4/3})$ bound.
Prior to our work, the best explicit constant (we know of) is the one derived from an argument of Sz\'ekely \cite{sz1}, which gives $u(n)\le 8n^{4/3}$ for all $n$.
Our main result is the following constant factor improvement.

\begin{theorem}\label{thm:main}
$u(n)\le\sqrt[3]{\frac{29}{4}}n^{4/3}=1.93...\cdot n^{4/3}$.
\end{theorem}

Our proof is based on a careful examination of Sz\'ekely's argument to get rid of a few extra factors, an improved multigraph crossing lemma\footnote{Note that in an earlier version of this paper, that appeared in the EuroCG '20 proceedings, we have proved the weaker bound $2.09n^{4/3}$; our new improvement is due to our new  multigraph crossing lemma.} and some simple observations about unit distance graphs.
In general, any improvement in the crossing lemma also improves our constant.
To be more precise, if $cr(G)\ge c\frac{m^3}{n^2}-O(n)$, then $u(n)\le\sqrt[3]{\frac{1}{4c}}n^{4/3}+O(n)$ holds.\footnote{A similar constant factor improvement was made by Pach et al.~\cite{prtt} for the Szemer\'edi--Trotter theorem.}

The rest of this paper is organized as follows.
In Section \ref{sec:crossing}, we discuss variants of the crossing lemma.
In Section \ref{sec:multigraphcrossing} we state and prove the multigraph crossing lemma we will use.
In Section \ref{sec:harmoniccrossing} we state and prove a harmonic crossing lemma and apply it to non-homotopic multigraphs.
In Section \ref{sec:proof} we prove our main result, Theorem \ref{thm:main}.
In Section \ref{sec:smalln} we examine the best bounds for $u(n)$ for small $n$.
In Section \ref{sec:1537} we prove $u(15)=37$, the first value that was not determined yet.
Finally, in Section \ref{sec:conc} we make some concluding remarks.

\section{The crossing lemma and its variants}\label{sec:crossing}

Draw a (not necessarily simple) graph in the plane so that vertices are mapped to points and edges to simple curves that do not go through the images of the vertices other than their endpoints', no three edges intersect at the same point, apart from vertices and no two edges have infinitely many common points. The \emph{crossing number of a graph} $G$, denoted by $cr(G)$, is defined as the minimum number of intersection points among the edges of $G$ in such drawings, counted with multiplicity.
The crossing lemma, which was first proved by Ajtai, Chv\'atal, Newborn, Szemer\'edi \cite{acnsz} and, independently, by Leighton \cite{l}, is that for any simple graph $cr(G)\ge \Omega(\frac {m^3}{n^2})$ if $m\ge 4n$, where $n$ is the number of vertices and $m$ is the number of edges.
The hidden constant has been improved several times; the current best is the following result.

\begin{lemma}[Ackerman \cite{a}]\label{newcrossinglemma}
If a simple graph has $n$ vertices and $m$ edges, then $cr(G)\ge \frac{m^3}{29n^2}-\frac{35n}{29}$.
Moreover, if $m\ge 6.95n$, then $cr(G)\ge \frac{m^3}{29n^2}$.
\end{lemma}

In the next two subsections we state two variants of the crossing lemma.

\subsection{The multigraph crossing lemma}\label{sec:multigraphcrossing}

In a multigraph, we allow several edges to connect the same pair of vertices.
These are called \emph{parallel} edges and the \emph{multiplicity} of an edge is the number of parallel edges to it (including itself).
Say that some parallel edges, $e_i$ ($i=1,...,j$) are \emph{close} if\\
(i) they do not intersect apart from their endpoints,\\
(ii) no vertices or crossings occur in the bounded region $S_i$ bordered by $e_i$ and $e_{i+1}$,\\
(iii) for any other edge $f$, all connected components of $f\cap Cl(S_i)$ connect $e_i$ and $e_{i+1}$.

Note that given any drawing and $j$ close edges $e_i$ ($i=1,...,j$), we can always draw a new edge $e_{j+1}$ that $e_i$ ($i=1,2,...,j+1)$ are close.

In this paper, we do not allow loops in multigraphs, unless stated otherwise. From now on, $\lbrace x\rbrace$ will denote the fractional part of $x$.\\

Let $cr(n,m)$ be the minimum of $cr(G)$ over simple graphs with $n$ vertices and $m$ edges, and $cr(n,m,k)$ be the minimum of $cr(G)$ over graphs with $n$ vertices and $m$ edges, each of which is with multiplicity at most $k$.
Note that $cr(n,m)=cr(n,m,1)$.

It has been long shown by Sz\'ekely \cite[Theorem~7]{sz1} that $cr(n,m,k)=\Omega(k^2\cdot cr(n,m/k))$ and this is best possible when $m\ge 4kn$, but the constant hidden in the $\Omega$ notation was larger than $1$. 
Some later proofs gave lower bounds that also worked (with a slight modification) for multigraphs without losing a further constant factor, as in \cite{prtt}, where it was shown in \cite[Theorem~3]{prtt} that $cr(n,m,1)\ge\frac{1}{31.1}\cdot\frac{m^3}{n^2}-1.06n$ and in \cite[Corollary~5.2]{prtt} that $cr(n,m,k)\ge\frac{1}{31.1}\cdot\frac{m^3}{kn^2}-1.06nk^2$. Although it is not noted there, the proof of Corollary 5.2 (and analogously, the proof of the earlier \cite[Theorem~4.3]{pt}) implies the following statement.

\begin{lemma}[Implicit in \cite{prtt,pt}]\label{prttmultigraphcrossinglemma}
For any function $f(n,m)$ for which $f(n,m)\le cr(n,m)$ and $f(n,m)$ is convex in $m$ for all fixed $n$, $cr(n,m,k)\ge k^2\cdot f\left(n,\frac{m}{k}\right)$ holds. A little more precisely, $cr(n,m,k)\ge k^2\cdot\left(\left(1-\left\lbrace\frac{m}{k}\right\rbrace\right)\cdot f\left(n,\left\lfloor\frac{m}{k}\right\rfloor\right)+\left\lbrace\frac{m}{k}\right\rbrace\cdot f\left(n,\left\lceil\frac{m}{k}\right\rceil\right)\right)$.
\end{lemma}

For comparison with the proof of Lemma \ref{multigraphcrossinglemma}, we include their proof here.

\begin{proof}
Analogously to \cite[Corollary~5.2]{prtt}, consider any drawing of any multigraph $G$ with $n$ vertices, $m$ edges and maximal edge multiplicity $k$, and take a simple subgraph $G'\subset G$ of the drawing of it in the following way.
If an edge has multiplicity $i$, we take at most one of the $i$ copies, all of them with probability $\frac{1}{k}$, and we make this decision for all edges independently.

This way, $G'$ has $n$ vertices, the expected number of its edges is $\frac{m}{k}$, and thus, the expected number of its crossings is at least $f\left(n,\frac{m}{k}\right)$ by Jensen's inequality, as $f$ is convex in $m$ by assumption. Since a crossing of $G$ is also a crossing in $G'$ with probability at most $\frac{1}{k^2}$ (for crossings of pairs of parallel edges of $G$ this probability is zero), the total number of crossings in the drawn version of $G$ is at least $k^2\cdot f\left(n,\frac{m}{k}\right)$.

To prove the slightly better lower bound, define $\bar f(x)=\left(1-\left\lbrace x\right\rbrace\right)\cdot f\left(n,\left\lfloor x\right\rfloor\right)+\left\lbrace x\right\rbrace\cdot f\left(n,\left\lceil x\right\rceil\right)$.
Since $cr(n,m)$ is defined only for integer values of $m$, the conditions of the lemma also hold for $\bar f$.
From our first lower bound, $cr(n,m,k)\ge k^2\cdot \bar f\left(n,\frac{m}{k}\right)=k^2\cdot\left(\left(1-\left\lbrace\frac{m}{k}\right\rbrace\right)\cdot f\left(n,\left\lfloor\frac{m}{k}\right\rfloor\right)+\left\lbrace\frac{m}{k}\right\rbrace\cdot f\left(n,\left\lceil\frac{m}{k}\right\rceil\right)\right)$, by the definition of $\bar f$.
\end{proof}

By Lemma \ref{newcrossinglemma}, $f(n,m)=\frac{m^3}{29n^2}-\frac{35n}{29}$ satisfies the conditions of Lemma \ref{prttmultigraphcrossinglemma}, which gives the following result.

\begin{corollary}\label{prttcorollary}
$cr(n,m,k)\ge \frac{m^3}{29kn^2}-\frac{35}{29}\cdot nk^2$.
\end{corollary}

This improved constant leads to (constant factor) improvements in several theorems in incidence geometry; see \cite{ps} and \cite[Theorem 8]{sz1}.

In the following, we give a slightly better lower bound on $cr(n,m,k)$ with an entirely different proof method.
Moreover, this enables us to determine the exact values of the function $cr(n,m,k)$ for all $k\vert m$ once the function $cr(n,m)$ is given.

\begin{lemma}\label{multigraphcrossinglemma}
(1) $k^2\cdot cr(n,\lfloor\frac{m}{k}\rfloor)\le cr(n,m,k)\le k^2\cdot cr(n,\lceil\frac{m}{k}\rceil)$.

(2) $cr(n,m,k)\ge k^2\cdot\left(\left(1-\left\lbrace\frac{m}{k}\right\rbrace\right)\cdot cr\left(n,\left\lfloor\frac{m}{k}\right\rfloor\right)+\left\lbrace\frac{m}{k}\right\rbrace\cdot cr\left(n,\left\lceil\frac{m}{k}\right\rceil\right)\right)$.
\end{lemma}

\begin{proof}
The upper bound in (1) follows from taking a simple graph with $n$ vertices, $\left\lceil\frac{m}{k}\right\rceil$ edges and $cr(n,\left\lceil\frac{m}{k}\right\rceil)$ crossings, and converting it into a multigraph with $m$ edges by taking (at most) $k$ copies of each edge, and drawing each copy close to the original edge.
This way in the multigraph there are (at most) $k^2$ times as many crossings as in the simple graph.

To prove the lower bounds, take a graph $G$ with $n$ vertices, $m$ edges and edge multiplicity at most $k$ such that $cr(G)=cr(n,m,k)$.
Take a drawing of $G$ in which the number of crossings is $cr(n,m,k)$.
First, we will change the drawing of $G$.

Suppose that there are two vertices for which the parallel edges between them are not all close.
Take an edge $e$ between them which has the least number of crossings with edges not parallel to it, and redraw all the other parallel edges so that they are close to $e$.
After the redrawing, edges parallel to $e$ do not cross each other, and the number of crossings with other edges did not increase.
So the new drawing has at most as many crossings as the original one.

With repeating this for each pair of vertices, if necessary, we can get a drawing of $G$ with $cr(n,m,k)$ crossings, and all parallel edges close.

In the next step, we will not only change the drawing, but even modify $G$.

Say that an edge is \emph{full} if its multiplicity is $k$, and call all $k$-tuples of parallel edges \emph{full}.
Suppose that $e_1$ is a non-full edge, $e_1$ and $e_2$ are not parallel, and 
$e_1$ has no more crossings than $e_2$.
Then we can redraw $e_2$ parallel to $e_1$ close to it.
This step does not increase the number of crossings, as all crossings between $e_1$ and $e_2$ are eliminated, and the number of crossings between $e_2$ and other edges did not increase.
In fact, since $cr(G)=cr(n,m,k)$ and the drawing we have is minimal with repsect to the number of crossings, we know that the number of crossings cannot decrease, so it remains the same.
But this implies that $e_1$ and $e_2$ must have the same number of crossings.
As this must hold for any pair of non-full edges, we can redraw any of them to be close.
If there are at least $k$ non-full edges, with such redrawings we can make $k$ of them full.

So using such redrawings, in the end we can obtain a graph $G'$ that has $cr(n,m,k)$ crossings and less than $k$ non-full edges, which are parallel to each other.
Denote by $G_1$ the simple graph obtained from $G'$ by taking only one edge of each full $k$-tuple of parallel edges.
Since any crossing in $G_1$ corresponds to $k^2$ crossings in $G'$, we proved the lower bound in (1).

Now, let $G_2$ be any maximal simple subgraph of $G'$ and consider its drawing inherited from the drawing of $G'$.
As we have seen earlier, if $k$ divides $m$, then $G_1=G_2$, and otherwise $E(G_2)\setminus E(G_1)$ has a single edge, $e$.
Since all edges of $G'$ parallel with $e$ have the same number of crossings (as they are all close), and this number is exactly $k$ times as many as the number of crossings of $e$ in $G_2$, $cr(G')=k^2\left(\left(1-\left\lbrace\frac{m}{k}\right\rbrace\right)\cdot cr(G_1)+\left\lbrace\frac{m}{k}\right\rbrace\cdot cr(G_2)\right)$. This implies (2).
\end{proof}

Using either Lemma \ref{prttmultigraphcrossinglemma} or Lemma \ref{multigraphcrossinglemma}, we can prove a special case that we will need for small $n$.

\begin{corollary}\label{crossing2}
If $G$ is a graph with $n\ge 3$ vertices and $m$ edges and all edges have multiplicity at most $2$, then $cr(G)\ge 2m-12n+24$ if $m$ is even.
\end{corollary}
\begin{proof}
This follows from plugging in $cr(n,s)\ge s-(3n-6)$ to
Lemma \ref{multigraphcrossinglemma}.
\end{proof}

Next we state the bound that we will use for general $n$, which is slightly stronger than Corollary \ref{prttcorollary}.

\begin{corollary}\label{multigraph2}
If $G$ is a multigraph with $n$ vertices and $m$ edges, where $\left\lfloor\frac{m}{k}\right\rfloor\ge6.95n$ and all edges have multiplicity at most $k$, then $cr(G)\ge\frac{m^3}{29kn^2}$,
a little more precisely $$cr(G)\ge k^2\cdot\left(\left(1-\left\lbrace\frac{m}{k}\right\rbrace\right)\cdot \frac{\left(\left\lfloor\frac{m}{k}\right\rfloor\right)^3}{29n^2}+\left\lbrace\frac{m}{k}\right\rbrace\cdot \frac{\left(\left\lceil\frac{m}{k}\right\rceil\right)^3}{29n^2}\right).$$
\end{corollary}

\begin{proof}
The second part directly follows from Lemma \ref{newcrossinglemma} and Lemma \ref{multigraphcrossinglemma} (2). The first part follows from the second part using Jensen's inequality.
\end{proof}

Note that this very slight improvement is due to the fact that we can use Lemma \ref{multigraphcrossinglemma} to functions that are not convex in $m$, or not even defined for all $m$, while we cannot use Lemma \ref{prttmultigraphcrossinglemma} to such functions. Thus, here we could use the second part of Lemma \ref{newcrossinglemma} directly, while in Corollary \ref{prttcorollary}, we could not do this. Although with a careful examination (using the lower convex envelope of the combination of the bound from the second part of Lemma \ref{newcrossinglemma} and some bound for small $m$), we could obtain a bound that is just as good as Corollary \ref{multigraph2} if $m$ is large enough. 

\begin{corollary}\label{multigraph3}
If $G$ is a graph with $n$ vertices and $m\ge 13.9n$ edges and all edges have multiplicity at most $2$, then $cr(G)\ge\frac{m^3}{58n^2}$ if $m$ is even.
\end{corollary}

\begin{proof}
This is a special case of Corollary \ref{multigraph2}.
\end{proof}

\bigskip

Finally, let us state a very recent
version of the crossing lemma for multigraph without homotopic edges \cite{ptt}.

Take a multigraph in which loops are allowed and call a certain drawing of such a graph in the plane a \emph{topological multigraph}. In a topological multigraph, two parallel edges, $e_1$ and $e_2$, between vertices $v_1$ and $v_2$ are called \emph{homotopic} if there exists a homotopy (a continuous funtction $f:[0,1]\times[0,1]\rightarrow\mathbb{R}^2$) between them that

1) takes $v_1$ for all values $(x,0)$ and $v_2$ for all values $(x,1)$

2) takes the curve $e_1$ parameterized by $y$ for $(0,y)$ and the curve $e_2$ for $(1,y)$

3) does not take any vertex of $G$ as $f(x,y)$ for $x,y\in\left(0,1\right)$.

Call a topological multigraph \emph{non-homotopic} if there are no two parallel edges, which are homotopic.

\begin{theorem}
[Pach, Tardos, T\'oth \cite{ptt}]\label{nonhomotopiccrossinglemma}
The crossing number of a non-homotopic topological multigraph $G$ with $n$ vertices and $m>4n$ edges satisfies $cr(G)\ge\frac{1}{24}\frac{m^2}{n}$.
\end{theorem}

Denoting by $cr^{non-h}(n,m)$ the mimimum of $cr(G)$ over all non-homotopic topological multigraphs with $n$ vertices and $m$ edges, Theorem \ref{nonhomotopiccrossinglemma} states that $cr^{non-h}(n,m)\ge \frac{1}{24}\frac{m^2}{n}$ if $m>4n$.
It was also shown in \cite{ptt} that this is not sharp, for every fixed $n\ge 2$ as $m$ grows $cr^{non-h}(n,m)/m^2\to \infty$.
From above, they have proved $cr^{non-h}(n,m)\le 30\frac{m^2}n\log^2 \frac mn$ for every $n\ge 2$ and $m>4n$.

\subsection{Harmonic crossing lemma}\label{sec:harmoniccrossing}

\begin{lemma}[Harmonic crossing lemma]\label{harmoniccrossinglemma}
Take a simple graph $G$ with $n$ vertices and $m$ edges drawn in the plane. If for a drawing of $G$ in the plane, $x(e)$ denotes the number of edges that cross edge $e$, then $\sum\limits_{e\in E}{\frac{1}{x(e)+1}}\le 3n-6$.
\end{lemma}

The lemma follows from the existence of an independent vertex set with $\sum_{v\in V(G)}{\frac{1}{deg(v)+1}}$ (Caro \cite{c}, Wei \cite{w}), applied to the graph whose vertices are the edges of $G$, connected if they cross.
For completeness, we present the full proof, stated for this special case.

\begin{proof}
Take a random order of the edges of $G$. Take a subgraph $G'$ of $G$ such that $V(G')=V(G)$ and an edge $e\in E(G)$ is in $E(G')$ exactly if $e$ precedes all edges crossing it according to the ordering we took. This selection process prevents all pairs of edges of $G'$ to cross, thus $G'$ is planar, meaning $\left\lvert E(G')\right\rvert\le 3n-6$ regardless of the ordering.

For any $e\in E(G)$, the probability of $e\in E(G')$ equals $\frac{1}{x(e)+1}$, as $e$ was chosen exactly if out of the $\left(x(e)+1\right)$-element set containing $e$ and the edges crossing $e$, $e$ comes first in the ordering. Thus, $\mathbb{E}\left(\left\lvert E(G')\right\rvert\right)=\sum\limits_{e\in E}{\frac{1}{x(e)+1}}$. This finishes the proof.
\end{proof}

Note that if $G$ is planar, then equality holds.

The name ``harmonic crossing lemma" refers to the fact that Lemma \ref{harmoniccrossinglemma} implies 
\[\frac{m}{3n-6}\le
\frac{m}{\sum\limits_{e\in E(G)}{\frac{1}{x(e)+1}}}\le\frac{\sum\limits_{e\in E(G)}{{x(e)+1}}}{m}\le
\frac{2cr(G)+m}{m}\]
where the middle inequality follows from the harmonic mean--arithmetic mean inequality for the numbers $x(e)+1$.
This gives $cr(G)\ge \frac{m^2}{6n-12}-\frac m2$, weaker than
Lemma \ref{newcrossinglemma}, but more widely applicable, as we will see.
The question naturally arises whether the upper bound in Lemma \ref{harmoniccrossinglemma} can be improved to something like $\frac{n^2}{m}$, but this is not the case, as any connected graph $G$ can be easily drawn such that all edges of a spanning tree $T$ of $G$ avoid all crossings.
Because of this $\sum\limits_{e\in E(G)}{\frac{1}{x(e)+1}}\ge \sum\limits_{e\in T}{\frac{1}{0+1}}=n-1.$
It is, however, likely that the constant 3 can be improved for large $m$.

Now we will show how to improve the constant of Theorem \ref{nonhomotopiccrossinglemma} using the Harmonic crossing lemma and Lemma 2 from \cite{ptt}, which states that any non-homotopic multigraph on $n$ vertices and no pairs of crossing edges has at most $3n-3$ edges. 

Now we will state an improved version of Theorem \ref{nonhomotopiccrossinglemma} based on the above:

\begin{theorem}
The crossing number of a non-homotopic topological multigraph $G$ with $n\ge 2$ 
vertices and $m$ edges satisfies $cr(G)\ge\frac{m^2}{6n-6}-\frac m2$.
\end{theorem}

\begin{proof}
By using the same argument as in the proof of \ref{harmoniccrossinglemma}, we can prove that $\sum\limits_{e\in E(G)}{\frac{1}{x(e)+1}}\le3n-3$, using \cite[Lemma 2]{ptt}.
From the inequality between the arithmetic and the harmonic mean, \[\frac{m}{3n-3}\le
\frac{m}{\sum\limits_{e\in E(G)}{\frac{1}{x(e)+1}}}\le\frac{\sum\limits_{e\in E(G)}{{x(e)+1}}}{m}\le
\frac{2cr(G)+m}{m}\]
which implies the statement.
\end{proof}

\section{Proof of Theorem \ref{thm:main}}\label{sec:proof}

Fix a UDG on $n$ vertices with $u(n)$ edges and the images of the vertices of one of its planar realizations; denote these $n$ points by $P$ and the UDG by $G$.
In the following, we do not differentiate between vertices and their images.
Note that due to the maximality of $G$, any two points at unit distance form an edge.

The statement is trivial for $n\le 2$, so we only have to prove it for $n\ge3$.

First, suppose that all vertices of $G$ have degree at least $3$. Now, similarly to Sz\'ekely \cite{sz1}, we can draw unit circles around all the vertices of $G$. Note that whenever two points are unit distance apart, their circles will be incident to one another. Divide the circles into circular arcs with the points of $P$ that fall on them. This way we obtain exactly $\sum\limits_{v\in G}{deg_G(v)}=2u(n)$ circular arcs.

Define a graph $H$ whose vertex set is the same as $G$'s, and its edges are these circular arcs.
For any pair of vertices there are at most two circles incident to both of them. Also, on a circle, any pair from $P$ is connected by at most one arc from $E(H)$ on the same circle, as at least $3$ vertices fall on each circle. Thus, the multiplicity of any edge is at most $2$.
So $H$ is a graph with $n$ vertices and $2u(n)$ edges, with edge multiplicity at most $2$.
In the above described drawing of $H$, there are at most $2\cdot\binom{n}{2}=n^2-n$ crossings, since all pairs of circles cross at most twice.

Now we have three cases:

1) If $u(n)\ge 6.95n$, then from Corollary \ref{multigraph3}, we get that $cr(H)\ge\frac{8u(n)^3}{58n^2}=\frac{4}{29}\cdot\frac{u(n)^3}{n^2}$, and thus $\frac{4}{29}\cdot\frac{u(n)^3}{n^2}\le n^2\Rightarrow u(n)\le\sqrt[3]{\frac{29}{4}}\cdot n^{4/3}$.

2) If $u(n)<6.95n$ and $n\ge47$, then $6.95n<\sqrt[3]{\frac{29}{4}}\cdot n^{4/3}$, so the statement is true.

3) If $n\le380$, then we can lower bound the intersections among these unit circles in the following way.
For a vertex $v$ with degree $deg(v)$, there are exactly $\left({deg(v)}\atop\vphantom{|}2\right)$ pairs of circles that intersect in $v$.
Therefore, we have $n^2-n\ge \sum_v \left({deg(v)}\atop\vphantom{|}2\right)\ge n\cdot\left({\sum_v deg(v)/n}\atop\vphantom{|}2\right)= u(n)\left(\frac{2u(n)}n-1\right)$ by Jensen's inequality.
An elementary calculation gives $u(n)<\sqrt[3]{\frac{29}{4}}\cdot n^{4/3}$ when $n\le 380$.\footnote{Note that in the later parts of our proof we could also reduce the upper bound on $cr(H)$ by $u(n)(\frac{2u(n)}n-1)$ using this argument, but it would not change its order of magnitude or effect the constant we obtain.}
Moreover, quite surprisingly, this simple bound (combined with a linear lower bound for the number of crossings in $H$) beats the best previous bound for small $n$ starting from $n=25$.
(These values can be found in Table 1.)

So the only case when $G$ can have more than $\sqrt[3]{\frac{29}{4}}\cdot n^{4/3}$ edges is if the assumption that all vertices have degree at least $3$ is false. Now suppose that $G$ has the smallest number of vertices among those UDGs that do not satisfy the upper bound. Then by removing a vertex with degree at most $2$, we can get a UDG with $n-1$ vertices and at least $u(n)-2$ edges. But since $\sqrt[3]{\frac{29}{4}}\cdot n^{4/3}-\sqrt[3]{\frac{29}{4}}\cdot (n-1)^{4/3}>2$ for $n\ge 3$,
this would mean that the obtained UDG on $n-1$ vertices also does not satisfy the upper bound, contradicting the assumption that $G$ was the smallest such UDG.

This finishes the proof of Theorem \ref{thm:main}.

\section{Best bounds for $n\le 30$}\label{sec:smalln}

In Table 1 we list the best known bounds, along with constructions.
For $n\le 14$, the exact values of $u(n)$ were known, established in the thesis of Schade \cite{sch}, while $u(15)=37$ is our contribution (see Section \ref{sec:1537}).

\begin{center}
\begin{tabular}[t]{ | c | c | c | }
\hline
$n$&$u(n)$&Lower bounding graph(s)\\
\hline
1&0&\begin{tikzpicture}[scale=0.7,line cap=round,line join=round]
\clip(-1.16,-0.22) rectangle (-0.85,0.14);
\begin{scriptsize}
\fill [color=black] (-1,0) circle (1.5pt);
\end{scriptsize}
\end{tikzpicture}\\
\hline
2&1&\begin{tikzpicture}[scale=0.7,line cap=round,line join=round]
\clip(-1.13,-0.31) rectangle (0.22,0.23);
\draw (-1,0)-- (0,0);
\begin{scriptsize}
\fill [color=black] (0,0) circle (1.5pt);
\fill [color=black] (-1,0) circle (1.5pt);
\end{scriptsize}
\end{tikzpicture}\\
\hline
3&3&\begin{tikzpicture}[scale=0.7,line cap=round,line join=round]
\clip(-1.11,-0.16) rectangle (0.11,0.99);
\draw (-0.5,0.87)-- (-1,0);
\draw (-0.5,0.87)-- (0,0);
\draw (-1,0)-- (0,0);
\begin{scriptsize}
\fill [color=black] (0,0) circle (1.5pt);
\fill [color=black] (-1,0) circle (1.5pt);
\fill [color=black] (-0.5,0.87) circle (1.5pt);
\end{scriptsize}
\end{tikzpicture}\\
\hline
4&$5^{*}$&\input{Figures/G4}\\
\hline
5&$7^{*}$&\input{Figures/G5}\\
\hline
6&$9^{*}$&\input{Figures/G6-1}\input{Figures/G6-2}\input{Figures/G6-3}\input{Figures/G6-4}\\
\hline
7&12&\input{Figures/G7}\\
\hline
8&1$4^{*}$&\input{Figures/G8-1}\input{Figures/G8-2}\input{Figures/G8-3}\\
\hline
9&18&\input{Figures/G9}\\
\hline
10&$20^{*}$&\input{Figures/G10}\\
\hline
11&$23^{*}$&\input{Figures/G11-1}\input{Figures/G11-2}\\
\hline

\end{tabular}
\end{center}

\begin{center}
\begin{tabular}{ | c | c | c | }
\hline
$n$&$u(n)$&Lower bounding graph(s)\\
\hline
12&27&\input{Figures/G12}\\
\hline
13&$30^{*}$&\input{Figures/G13}\\
\Xhline{5\arrayrulewidth}
14&$33^{*}$&\input{Figures/G14-1}\input{Figures/G14-2}\\
\hline
15&{\bf 37}&\input{Figures/G15}\\
\hline
16&41 or $42^{*}$&\input{Figures/G16}\\
\hline
17&43--47&\input{Figures/G17}\\
\hline
18&46--52&\input{Figures/G18}\\
\hline
19&50--$57^{*}$&\input{Figures/G19-1}\input{Figures/G19-2}\\
\hline
\end{tabular}
\end{center}

\begin{center}
\begin{tabular}{ | c | c | c | }
\hline
$n$&$u(n)$&Lower bounding graph(s)\\
\hline
20&54--63&\input{Figures/G20}\\
\hline
21&57--$68^{*}$&\input{Figures/G21}\\
\hline
22&60--{\bf 72}&\input{Figures/G22}\\
\hline
23&64--{\bf 77}&\input{Figures/G23}\\
\hline
24&68--{\bf 82}&\input{Figures/G24}\\
\hline
25&72--{\bf 87}&\input{Figures/G25-1}\input{Figures/G25-2}\\
\hline
26&76--{\bf 92}&\input{Figures/G26}\\
\hline
\end{tabular}
\end{center}

\begin{center}
\begin{tabular}{ | c | c | c | }
\hline
$n$&$u(n)$&Lower bounding graph(s)\\
\hline
27&81--{\bf 97}&\input{Figures/G27}\\
\hline
28&85-{\bf 102}&\input{Figures/G28-1}\input{Figures/G28-2}\\
\hline
29&{\bf 89--108}&\input{Figures/G29}\\
\hline
30&{\bf 93--113}&\input{Figures/G30}\\
\hline
\end{tabular}
\bigskip
\label{table1} \captionsetup{font=scriptsize}
\captionof{table}{Bounds for the maximal number of unit distances $u(n)$ among $n$ planar points.}
\end{center}

For $n\le 13$, the drawn graphs are known to be the only maximal UDGs.\footnote{Note that in \cite{bmp} it is incorrectly stated also for $n=14$ that the constructions were proved to be unique.}
In general, the upper bounds can be obtained using the inequality $u(n)\le \lfloor{n\over n-2}\cdot u(n-1)\rfloor$ (for $n\ge 3$), which is true because the edge density of the maximal UDGs with $n$ vertices is monotonically decreasing in $n$ as all subgraphs of a UDG are also UDGs.
This was observed by Schade, who sometimes also applied additional tricks---the values where such tricks are needed are denoted by a star.

For $n\ge 15$, the graphs that attain the lower bounds are also by Schade, with the exception of the ones for $n=29,30$, and the second graphs for $n=19,28$, which are our constructions based on graphs by Schade. The upper bounding values from $n\ge 22$ are also our improvements (the improved values are marked in bold), derived from the following refinement of the inequality $n^2-n\ge \sum_v \left({deg(v)}\atop\vphantom{|}2\right)$.

\begin{proposition}\label{prop:improvedbound}
If $G$ is a unit distance graph with $u(n)$ edges, then
\[
    n^2-n \ge 4\cdot|E(G)|-12n+24+\sum_{v\in V(G)} \left({deg(v)}\atop\vphantom{|}2\right)\]
    \[\ge 4\cdot u(n)-12n+24+n\cdot\left(1-\left\lbrace{\frac{2u(n)}{n}}\right\rbrace\right)\cdot\left({\left\lfloor\frac{2u(n)}{n}\right\rfloor}\atop\vphantom{|}{2}\right)+n\cdot\left\lbrace{\frac{2u(n)}{n}}\right\rbrace\cdot\left({\left\lceil\frac{2u(n)}{n}\right\rceil}\atop\vphantom{|}{2}\right).
\]

\end{proposition}

\begin{proof}
Recall that $n^2-n\ge cr(H)+\sum_{v\in V(G)} {\left({deg_G(v)}\atop\vphantom{|}2\right)}$ where $H$ is the graph obtained from $G$, as described in Section \ref{sec:proof}, with edges of multiplicity at most two.\footnote{We can again suppose that $G$ has no vertex of degree two, as in that case we would have $u(n)\le u(n-1)+2$, which would give a better upper bound for each $n>10$ in the table.}

For non-negative integers $n_1,...,n_k$ with a fixed sum, $\sum_{v\in V(G)}{\left(n_i\atop\vphantom{|}2\right)}$ is minimal if $\max_{1\le i,j\le k}\left({n_i-n_j}\right)\le 1$ (the proof is straightforward: $n_i-n_j\ge 2$ would mean $\left({n_i}\atop\vphantom{|}2\right)+\left({n_j}\atop\vphantom{|}2\right)>\left({n_i}\atop\vphantom{|}2\right)-\left(n_i-1\right)+\left({n_j}\atop\vphantom{|}2\right)+n_j=\left({n_i-1}\atop\vphantom{|}2\right)+\left({n_j+1}\atop\vphantom{|}2\right)$, a contradiction).

Since $\sum_{v\in V(G)}{{deg_G(v)}}=2u(n)$, we obtain $\sum_{v\in V(G)}{\left({deg_G(v)}\atop\vphantom{|}2\right)}\ge n\cdot\left(1-\left\lbrace{\frac{2u(n)}{n}}\right\rbrace\right)\cdot\left({\left\lfloor\frac{2u(n)}{n}\right\rfloor}\atop\vphantom{|}{2}\right)+n\cdot\left\lbrace{\frac{2u(n)}{n}}\right\rbrace\cdot\left({\left\lceil\frac{2u(n)}{n}\right\rceil}\atop\vphantom{|}{2}\right)$ using the same discrete version of Jensen's inequality as in the last paragraph of the proof of Lemma \ref{prttmultigraphcrossinglemma} and applying Corollary \ref{crossing2} to $H$ finishes the proof of the statement of Proposition \ref{prop:improvedbound}.
\end{proof}

As can be seen, the upper bounds diverge quite fast from the lower bounds.
For $n\ge521$, Theorem \ref{thm:main} gives the best upper bound.

\section{On 15 vertices 37 edges is best}\label{sec:1537}

We prove that $u(15)=37$, i.e., among 15 points in the plane, there can be at most 37 unit distances.
The lower bound follows from the construction found by Schade \cite{sch}, see Table 1.
The proof of the upper bound will follow from a straightforward case analysis.
Let us assume that $G_{15}$ is a UDG with 15 vertices and at least 38 edges.

\begin{proposition}
    $G_{15}$ has $38$ edges, $14$ degree-$5$ vertices and $1$ degree-$6$ vertex.
\end{proposition}
\begin{proof}
    The bound on the number of edges follows from considering the edge densities: $u(15)\le \lfloor{15\over 13}\cdot u(14)\rfloor<39$.
    
    If $G_{15}$ had a vertex whose degree is at most 4, then deleting this vertex would leave a UDG on $14$ vertices with at least $34$ edges, contradicting $u(14)=33$.
    But since the sum of the degrees is $76$, this is only possible as $76=14\cdot 5+6$.
\end{proof}

We will make the following observation:

Observation 1. \emph{In such a graph, all pairs of vertices with distance $1$ form an edge.}
\begin{proof}
Otherwise there would be a UDG with $15$ vertices and $39$ edges, which contadicts $u(15)<39$.
\end{proof}

The rest of the proof will not build on Schade's results.

Denote the degree-6 vertex of $G_{15}$ by $o$, the set of its $6$ neighbors by $N$ and the remaining 8 vertices by $R$. Now the following observations hold:

Observation 2. \emph{Any vertex from $R$ can have at most two neighbors from $N$.}
\begin{proof}
Otherwise along with $o$ they would form a $K_{2,3}$, which is not a UDG.
\end{proof}

This implies the following.

Observation 3. \emph{There are at most $16$ edges between $R$ and $N$.}

Observation 4. \emph{Any pair of vertices from $N$ can have at most one vertex from $R$ as their common neighbor.}
\begin{proof}
Otherwise they would form a $K_{2,3}$ along with $o$.
The sum of the degrees of the vertices in $N$ is $30$. There are 6 edges that go from $N$ to $o$ and from Observation 3, at most $16$ that go from $N$ to $R$.
\end{proof}

This implies the following.

Observation 5. \emph{There are at least $\frac{30-6-16}2=4$ edges among the vertices of $N$.}

We denote the neighborhood graph of $o$ by $G[N]$. Similarly to the above,

Observation 6. \emph{The number of edges between $R$ and $N$ is $30-6-2\left\lvert E(G[N])\right\rvert$.}

Since all the points of $N$ lie on a unit circle, this means that all degrees in $G[N]$ are at most $2$, so from Observation 5, $4\le \left\lvert E(G[N])\right\rvert\le 6$. But $5$ edges are not possible either: this could occur only as a $P_6$ (where we denote by $P_i$ the path on $i$ vertices) but it would imply that the vertices in $N$ form a regular hexagon with side-lengths $1$, thus forming $6$ edges because of Observation 1.
This leaves just $4$ options for what $G[N]$ can be: $C_6, P_5~\dot\cup~P_1, P_4~\dot\cup~P_2$, or $P_3~\dot\cup~P_3$.

Combining Observations 2 and 6, we get the following.

Observation 7. \emph{If $G[N]$ has 4 edges, then each vertex from $R$ has exactly two neighbors from $N$.}

We call the two edges leading from an $r\in R$ to $N$ a \emph{cherry}.
Recall that there can be at most one cherry on any two points of $N$, otherwise we would have a $K_{2,3}$ with $o$ and the other two vertices of the cherries.
The rest of the proof is a case analysis based on what $G[N]$ is.\\

Case $G[N]=C_6$: (See Figure \ref{15-38-C6} left.) Denote the vertices of $G[N]$ in circular order by $v_1,\ldots, v_6$. From Observation 6, there are $12$ edges from $N$ to $R$, so by the pigeonhole principle there are two vertices in $N$ that have a common neighbor. These must be adjacent, so without loss of generality suppose $r\in R$ is adjacent to $v_1$ and $v_2$. This $r$ cannot be adjacent to any of $v_3,\ldots,v_6$, moreover, it cannot even have a common neighbor with any of them (apart from $v_1$ and $v_2$): its distance from $v_3$ and $v_6$ is $2$, so $r$ can have at most one common neighbor with them, which is $v_2$ and $v_1$, respectively, while from $v_4$ and $v_5$, $r$ has distance more than $2$, thus it cannot have any common neighbors with them. And since $v_3$ has at most one common neighbor from $R$ with $v_4$ and no common neighbor from $R$ with $v_5$ or $v_6$, the $8$ edges connecting $\left\lbrace v_3,...,v_6\right\rbrace$ with $R$ belong to at least $5$ vertices from $R$. This leaves only $2$ potential neighbors for $r$ from $R$, contradicting that its degree is $5$.\\

\begin{center}
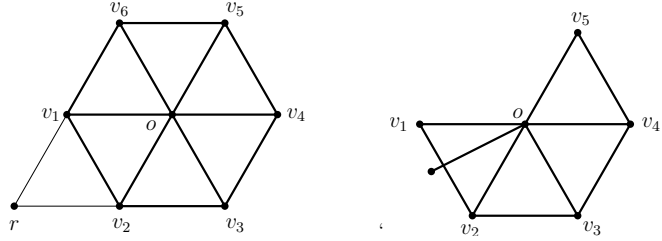

\scalebox{0.7}{\input{Figures/15-38-C6}~~~~~~~~~~~`~\input{Figures/15-38-P5-P1}}
\captionof{figure}{$G[N]=C_6$ to the left and $G[N]=P_5~\dot\cup~P_1$ to the right. Edges of $G[N]$ are denoted by thick, the edges between $N$ and $R$ are denoted by thin lines.}\label{15-38-C6}
\label{15-38-P5-P1}
\end{center}

Case $G[N]=P_5~\dot\cup~P_1$: (See Figure \ref{15-38-P5-P1} right.) Denote the vertices of the path $P_5$ in order by $v_1,\ldots, v_5$. Vertex $v_1$ has three neighbors from $R$, and from Observation 7, each of these has exactly one other neighbor from $N$. Therefore $v_1$ is the endpoint of 3 cherries, but it cannot be in a cherry with any of $v_3,v_4,v_5$ (the third vertex of the cherry with $v_5$ would be adjacent to $o$), a contradiction.\\

\begin{center}
\scalebox{0.75}{\input{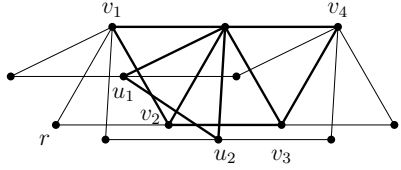}}
\captionof{figure}{$G[N]=P_4~\dot\cup~P_2$ is denoted by thick lines, the edges between $N$ and $R$ are denoted by thin lines.}
\label{15-38-P4-P2}
\end{center}

Case $G[N]=P_4~\dot\cup~P_2$: (See Figure \ref{15-38-P4-P2}.) Denote the vertices of $P_4$ in order by $v_1,\ldots, v_4$ and the vertices of $P_2$ by $u_1,u_2$. By Observation 7, vertex $v_1$ should be the endpoint of 3 cherries, but it cannot be in a cherry with any of $v_3,v_4$, so it forms a cherry with each of $v_2,u_1,u_2$. Similarly, $v_4$ forms a cherry with each of $v_3,u_1,u_2$. Denote by $r$ the common neighbor of $v_1$ and $v_2$ from $R$. It is easy to check that $r$ cannot be adjacent to any of the other 5 points of $R$ forming a cherry defined above: the neighbors of $v_4$ are too far from $r$, while the point at unit distance from $v_1$ and $r$, that is not $v_2$, is too far from the other vertices in $N$. But this leaves only 2 potential neighbors for $r$ from $R$, contradicting that its degree is 5.\\

Case $G[N]=P_3~\dot\cup~P_3$: (See Figure \ref{fig:P3P3}.) Denote the vertices of the $P_3$'s in order by $v_1,v_2,v_3$ and $u_1,u_2,u_3$. For all $(i,j)\in\lbrace (1,2),(2,3)\rbrace$ call the common neighbor of $v_i$ and $v_j$, other than $o$ (if exists), $v_{ij}$, and the common neighbor of $u_i$ and $u_j$, other than $o$ (if exists), other than $o$ $u_{ij}$. Also, for all $i,j\in\lbrace 1,2,3\rbrace$ call the common neighbor of $v_i$ and $u_j$, $w_{ij}$ (if exists). From Observation 7, all the vertices of $R$ are neighboring two vertices of $N$, and $v_1$ and $v_3$ do not have a common neighbor outside $o\cup N$, similarly to $u_1$ and $u_3$, so all the vertices of $R$ are from the above defined points.

\begin{center}
\scalebox{0.8}{\input{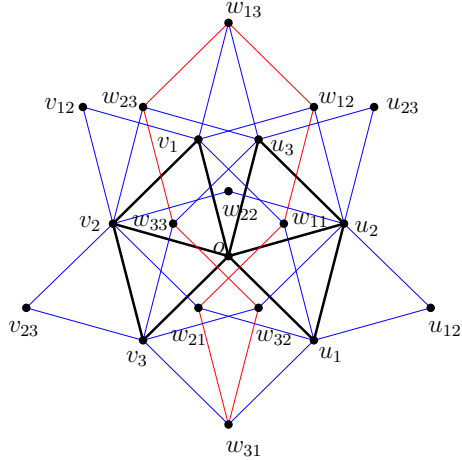}}
\captionof{figure}{$G[N]$ is denoted by thick black lines, the edges defining the $v_{ij}$'s, $u_{ij}$'s and $w_{ij}$'s are denoted by blue lines and $C$ is denoted by red lines.}
\label{15-38-P3-P3}\label{fig:P3P3}
\end{center}

Now suppose that $v_{12}$ and $w_{ij}$ are neighbors in $G_{15}$ for some $i,j\in\lbrace 1,2,3\rbrace$. Then $i$ cannot be $3$, as then $w_{ij}$ would be a common neighbor of $v_{12}$ and $v_3$, but their only common neighbor is $v_2$. If $i=1$ or $2$, then again $w_{ij}$ is the common neighbor of $v_i$ and $v_{12}$ and has distance $2$ from $o$, thus only touching the unit circle around it in $v_i$, which contradicts to it neighboring $u_j$. So $v_{12}$ cannot have any $w_{ij}$ as a neighbor. Applying the same argument to $v_{23}$, $u_{12}$ and $u_{23}$, we get that if any of them is inside $R$, then they could only be each others neighbors, and since $R$ is $3$-regular, they all exist and are all neighboring, thus forming a $K_4$, which is not a UDG, thus leading to a contradiction. So none of these four points are in $R$.

So $R$ consists of exactly $8$ of the $9$ $w_{ij}$'s and since the degree of $v_2$ and $u_2$ towards $R$ is $2$, while for all the other vertices, it is $3$, $w_{22}$ is the one missing. There is a cycle $C$ through all the points of $R$ (see Figure \ref{15-38-P3-P3}),  in the order $w_{11}, w_{12}, w_{13}, w_{23}, w_{33}, w_{32}$, $w_{31}$, $w_{21}$, by using the equalities $\overrightarrow{ow_{ij}}=\overrightarrow{ov_i}+\overrightarrow{ou_j}$. For example, $\overrightarrow{w_{11}w_{12}}=\overrightarrow{ou_1}-\overrightarrow{ou_2}=\overrightarrow{u_1u_2}$, which is a unit vector.

Every vertex in $R$ has exactly one more neighbor in $R$, these give $4$ diagonals of $C$, forming a perfect matching among its vertices. As $\left|w_{11}w_{13}\right|=\left|w_{13}w_{33}\right|=\left|w_{33}w_{31}\right|=\left|w_{31}w_{11}\right|=\sqrt{3}$, the diagonals starting from $w_{11}$, $w_{13}$, $w_{33}$ and $w_{31}$ must connect them with their $3$rd or $4$th neighbors. Since $w_{11}w_{13}w_{33}w_{31}$ forms a rhombus of side length $\sqrt{3}$, at most one of its diagonals can have length $1$. So we can suppose without loss of generality that $w_{13}$ is connected to $w_{32}$.
But then $w_{13}$, $w_{33}$, $w_{23}$, $u_3$ and $w_{32}$ form a $K_{2,3}$, a contradiction.\\

This finishes the proof of all cases.

\begin{remark}
    With similar methods one can show that if there is a UDG on 16 vertices with 42 edges, then it cannot have a vertex whose degree is at least 7.
    As $u(15)=37$, this implies that such a graph would contain 12 vertices of degree 5 and 4 vertices of degree 6.
\end{remark}

\section{Conclusion}\label{sec:conc}

We have improved the best known bound for the number of unit distances among $n$ points in the plane with a constant factor.
Though we have not explicitly stated, our proofs also work for bounding the number of unit distances among $n$ points on a sphere, as does the crossing lemma.
In case of general spheres the best lower bound for the number of unit distances is $\Omega(n\sqrt{\log n})$ \cite{swv}, while for the sphere of radius $1/\sqrt 2$ the lower bound is $\Omega(n^{4/3})$ \cite{ehp}.
This sphere is special because the distance of two points of the sphere is 1 if and only if the vectors from the center to them are perpendicular to each other.
Note that for this radius, and no other, two unit circles around antipodal points coincide, and there are more than two unit circles through a pair of antipodal points.
So for $1/\sqrt 2$ our proof only works with a slight modification, but with the same constant. If we also suppose that no two vertices are antipodal, then we get an even better bound, $\frac{\sqrt[3]{29}}{2}\cdot n^{4/3}+O(n)$, which is less than $1.54n^{4/3}+O(n)$.

The construction of \cite{ehp} is based on converting a set of $n$ points and $n$ lines in the plane with $I$ incidences to a set of $2n$ points on the sphere (none of them antipodal) with $I$ unit distances.
It was proved by Szemer\'edi and Trotter \cite{SzT} that the maximum of $I$ is $O(n^{4/3})$ and this is sharp.
The best known constant factor bounds are $0.42n^{4/3}< I < 2.5n^{4/3}$ by \cite{pt} and \cite{prtt}, respectively.\footnote{In fact, the upper bound since \cite{prtt} was reduced to $2.44n^{4/3}$ due to the improvement \cite{a} of the crossing lemma.}
From our proof and the above argument of \cite{ehp}, we only get a weaker upper bound for $I$.
But the lower bound $0.42n^{4/3}$ also shows the limitation of our method.

It would be natural to look at the same problem in the space.
This was also first studied by Erd\H os \cite{erdos3d}, who proved the bounds $\Omega(n^{4/3}\log\log n)$ and $O(n^{5/3})$, by using the grid as a construction and by noticing that $K_{3,3}$ is forbidden, while the best current upper bound is $O(n^{3/2})$ \cite{kmss,z}.

Finally, to make progress towards a $o(n^{4/3})$ bound, it would be interesting to study the extremal Tur\'an-number of UDG.
Denote by $\mathrm{ex}_u(n,F)$ the number of edges a UDG on $n$ vertices can have without a subgraph isomorphic to a graph $F$. What do we know about this function for various $F$? The only result we are aware of is the simple $\mathrm{ex}_u(n,F)=\Theta(u(n))$ if $\chi(F)>2$ and $\mathrm{ex}_u(n,C_4)\ge n^{1+c/\log\log n}$.\footnote{Zolt\'an L. Nagy, personal communication 2020.}
Another natural question is to study the possible number of occurrences of some fixed graph $F$ in a UDG.
Denoting this by $\mathrm{ex}_u(\# F,n)$, the only result we are aware of is again the easy $n^{1+c/\log\log n}\le \mathrm{ex}_u(\# K_3, n)\le 2u(n)$ and the recent results \cite{fk,pss} on $\mathrm{ex}_u(\# P_k, n)$.
For this and more related results, see \cite{bmp}.
We would also like to remark that forbidden subgraphs of UDG are also systematically studied, see \cite{cm,gp}.

\subsubsection*{Acknowledgement.}

The main result was obtained while working on the Polymath16 project about the Hadwiger--Nelson problem and is related, but not directly connected to it.

We would like to thank G\'eza T\'oth for useful discussions about the multigraph crossing lemma, N\'ora Frankl for discussions about $\mathrm{ex}_u(\# F,n)$, Heiko Harborth for sending us Schade's thesis, and several anonymous referees for the improvement of the presentation and for independently verifying some numerical bounds.

{}
\end{document}